\documentclass{article}

\usepackage{cite}
\usepackage{amsmath,amsfonts}
\usepackage{amssymb}
\usepackage{tcolorbox}
\usepackage{amsthm}
\usepackage{graphicx}
\graphicspath{{./figs/}}
\usepackage{fixltx2e}
\usepackage{array}
\usepackage{mathtools}
\usepackage{comment}
\usepackage{tikz,float,upgreek, subcaption}
\usetikzlibrary{decorations.pathreplacing} 
\usetikzlibrary{arrows} 

\numberwithin{equation}{section}

\newtheorem{theorem}{Theorem}
\newtheorem{lemma}{Lemma}
\newtheorem{corollary}[lemma]{Corollary}
\newtheorem{definition}[lemma]{Definition}
\newtheorem{remark}[lemma]{Remark}
\newtheorem{proposition}[lemma]{Proposition}
\newtheorem{problem}[lemma]{Problem}

\numberwithin{lemma}{section}

\begin{document}
\title{Solution to the isoperimetric $n$-bubble problem on $\mathbb{R}^1$ with log-concave density}
\author{John Ross}

\maketitle

\begin{abstract}

\end{abstract}

\section{Introduction}

The isoperimetric problem is a classic problem in mathematics that dates back to antiquity. The original problem, as well as its (equivalent) dual problem, are as follows:

\begin{problem}[The Classic Isoperimetric Problem]
Out of all simple closed curves $\gamma \subset \mathbb{R}^2$ with a fixed length of perimeter, find the curve that maximizes its enclosed area.
\end{problem}


\begin{problem}[The Dual Isoperimetric Problem]
Out of all simple closed curves $\gamma \subset \mathbb{R}^2$ which enclose a fixed area, find the curve that minimizes its length of perimeter.
\end{problem}

That the circle is the solution to the isoperimetric problem has been assumed knowledge for centuries. A modern approach to this problem was first developed by Steiner in the 1830's \cite{Steiner1838}, \cite{Steiner2013}. For a brief history and overview of the classic problem, we refer to \cite{Blasjo05}. 

There are innumerable directions in which the isoperimetric problem has been generalized, and we introduce several below. Perhaps the most immediate generalization is to higher dimensional spaces, where we $\gamma^n \subset \mathbb{R}^{n+1}$ or $\gamma^n \subset \Omega^{n+1}$ (with $\Omega^{n+1} \subseteq \mathbb{R}^{n+1}$ a subset of Euclidean space, and with boundary $\partial \Omega$ that might intersect $\gamma$. Here, solutions take on the familiar name of ``bubbles'' because soap bubbles and soap film naturally model the isoperimetric problem in $\mathbb{R}^3$.

The next immediate direction in which to generalize the problem is in working with two (or more) regions. In this problem (which we formulate here for area), we can look to (unions of pieces of) hypersurfaces that enclose two fixed volume measurements $M_1$, $M_2$, and seek to minimize total surface area. A (modern classic) result in this direction is the solution to the double bubble conjecture. This conjecture was believed solved until more modern times, then reverted to conjecture, and then eventually solved by Hutchings, Morgan, Ritore, and Ros \cite{HutchingsMorganRitoreRos02}. 
Anecdotally, the problem becomes more challenging as more regions are introduced. The triple bubble problem has been proven only in the case of the plane by Wichiramala \cite{Wichiramala04}. A standard 3-bubble has been suggested for higher dimensions, although it has not been proven to be uniquely area minimizing.

Our third method of generalizing the problem is to introduce a density function on the background space. In this problem, we introduce a background density function $f$ on our ambient space ($\mathbb{R}^n$ or $\Omega$) that affects how ``enclosed volume'' and ``surface area'' are measured. Specifically, for an ambient space $\Omega^n \subseteq \mathbb{R}^n$, we introduce a density function $f: \Omega \rightarrow [0,\infty)$. Then, for a hypersurface $\gamma^{n-1}$ enclosing an $n$-dimensional solid $\Gamma$ (so $\partial \Gamma = \gamma$), we define the (density-weighted) volume and surface area to be
\begin{align*}
\text{Weighted Volume} &= \int_\Gamma f\, d\mathcal{H}^n\\
\text{Weighted Surface Area} &= \int_\gamma f\, d\mathcal{H}^{n-1}
\end{align*}
where the integrals are taken with respect to the usual Hausdorff measure.

This work will focus on multi-bubble isoperimetric problems with a density, and with the ambient space being the real number line $\mathbb{R}^1$. Such problems have been studied under a variety of densities. By an argument in \cite{MorganPratelli13}, it is known that a perimeter-minimizing $n$-bubble solution will exist on $\mathbb{R}^m$ when using a density function that radially increases to infinity. On $\mathbb{R}^1$, some early results for a single bubble under a large class of density functions were first identified in \cite{BayleCaneteMorganRosales06}. 


Two families of density functions that have been studied recently are the $|x|^p$ functions and the log-convex functions. An initial exploration on density functions of the form $|x|^p$ was done by a research team under the mentorship of Frank Morgan \cite{HuangMorgan19}. This team identified solutions for the single and double bubble on these densities. Work was expanded to the 3-bubble and the 4-bubble in the case of $p=1$ in \cite{RossSCOPE}. Of note in these results is that any isoperimetric solution maintains a regular structure, regardless of the relative size of the bubble regions.

A radically different family of density functions on $\mathbb{R}^1$ was studied by Bongiovanni et al. \cite{ChambersBongiovanniETAL18}, who looked at single and double-bubble solutions under log-convex density functions. This approach was motivated by Gregory Chamber's proof \cite{Chambers2019} of the single-bubble isoperimetric solution in general $\mathbb{R}^n$ with log-convex density. 
In \cite{ChambersBongiovanniETAL18}, it was found that a double-bubble solution under a log-convex density function did not have a standard structure. Instead, the solution could either be made up of two intervals (one per region) or three intervals (with one region sandwiched between two components of the other region), and the isoperimetric structure depended on the sizes of the two masses in questions. In recent work, \cite{Sothanaphan20} extended these results to identify possible triple-bubbles.

In both Huang et al. and Bongiovanni et al.'s work, radially symmetric, increasing density functions were studied. 
This paper looks at symmetric, radially increasing density functions $f$ that satisfy an additional property of log-concavity. This property stipulates that $\left[ \log f \right]'' \leq 0$. This study extends the work of \cite{HuangMorgan19} and \cite{RossSCOPE}, as $|x|^p$ is a log-concave function. Specifically, our work address the following problem:

\begin{problem}[Our Isoperimetric Problem]
Consider the real number line $\mathbb{R}^1$, imbued with a density function $f$ that satisfies the following:
\begin{itemize}
\item $f$ is radially symmetric ($f(x) = f(-x)$)
\item $f$ has a point of zero density at the origin ($f(0) = 0$)
\item $f$ is radially increasing ($f'(x) \geq 0$ if $x > 0$)
\item $f$ is log-concave ($\left[\log(f(x))\right]'' < 0$ for $x \neq 0$)
\end{itemize}
Then, given a set of $n$ masses $M_1 \leq \dots \leq M_n$, find a configuration of $n$ regions with these weighted masses and with minimal weighted perimeter.
\end{problem}

We will formally state the solution to this problem in Theorem \ref{TheoremNBubble} below. Unlike Bongiovanni et al.'s work with log-convex functions, we will find that a general structure is upheld for all $n$ and for all $n$-bubbles. Furthermore, our isoperimetric solution is unique up to reflective symmetry across the origin, and  our proof gives the solution for an arbitrary number of regions. We will call this $n$-bubble structure the ``standard position'' for $n$ masses, and formally define/introduce it in Section \ref{Section33}.





Of special note is that the density functions we study have a point of 0 density at the origin. This is not inherent in the definition of ``log-concave'', but it is necessary for our arguments. The point will serve as an ``anchor'' for endpoints of intervals, and will keep our entire configuration from ``sliding'' as we tweak the size of certain masses. More specifically, we will use a first-variation argument to show that moving interval endpoints at speed $1/f$ will naturally preserve enclosed area while moving perimeter. This first variation argument breaks down when $f = 0$, making the origin a natural interval endpoint. Our log-concave requirement will also play a role in our first variation analysis.

\subsection{Overview of the argument}



Our argument follows a proof by induction on the number of regions. The first several cases are already known in the case of $|x|^p$ (see \cite{HuangMorgan19} \cite{RossSCOPE}), and for a general log-concave density function the base case is given as an immediate offshoot of Corollary \ref{CorOneIntervalPerRegion}. For the inductive case, we assume that the standard position is the only isoperimetric solution for bubbles of $n$ or fewer regions. For contradiction, we assume there exist a specific set of $n+1$ masses for which an isoperimetric solution \emph{other} than standard position exists. We spend several sections drawing conclusions about what any isoperimetric configuration (standard or nonstandard) must look like. An important result here is Corollary \ref{CorOneIntervalPerRegion}, in which we use the first variation formula to show that an $n$-bubble must have exactly $n$ intervals. (Note that this immediately distinguishes our log-concave situation from the log-convex case studied in \cite{ChambersBongiovanniETAL18}.) Because of this, we can differentiate the standard and nonstandard solutions by identifying the ``outermost interval that breaks the standard pattern.'' We then use this interval as the foundation for our contradiction, swapping two intervals to get less total perimeter. In order to make the comparison work, we first carefully adjust the size of certain masses to develop a new set of masses with two isoperimetric solutions (one standard, one nonstandard). that can be easily compared, and from which the contradiction more readily appears. The mass-adjustments are described in Section 4, and a fundamental result that allows us to adjust these masses is Proposition \ref{PropInequality2}.

Our paper is organized as follows. Section 2 includes definitions and other preliminary work. Section 3 explores the first variation formulas for mass and perimeter, and then uses these tools to show any isoperimetric $n$-bubble must have exactly $n$ intervals. Section 4 explores how tweaking our region sizes (by inflating or deflating certain masses) will change our isoperimetric configurations. Finally, Section 5 shows our induction proof of the $n$-bubble problem.

\subsection{A quick remark on images in this paper}

Throughout this paper, we inlude images of the real number line with a density function. The images almost exclusively show the density function $f(x) = |x|$. However, this is simply done to standardize the images: the function shown should be viewed as a stand-in for any radially symmetric, log-concave density function with a point of 0 density at the origin. The exception is Figure \ref{Figure1Example}, which gives a basic example specifically for $f(x) = |x|$.

\section{Preliminaries}

In this section we include some basic definitions and early results.

\begin{definition}
A \textbf{density} function on $\mathbb{R}$ is simply a nonnegative function $f$. 
\end{definition}

\begin{definition} Given an interval $[a,b]$ and a density $f$, the \textbf{weighted mass} of the interval with respect to the density is $\int_a^b f$.  The \textbf{weighted perimeter} with respect to the density is defined to be $f(a) + f(b)$. Note that, in this paper, weighted mass is sometimes simply called mass, area, or volume, and that weighted perimeter is sometimes called perimeter.
\end{definition}

\begin{definition}
A \textbf{region} on $\mathbb{R}$ is a collection of disjoint intervals $[a_i, b_i]$ with  total mass $M$ and perimeter $P$
    \begin{align*}
        M &= \displaystyle\sum_i \int_{a_i}^{b_i} f\\
        P &= \displaystyle\sum_i f(a_i) + f(b_i)
    \end{align*}
\end{definition}
We remark that, under many natural density functions (in particular: density functions that have a positive lower bound), regions with finite mass or finite perimeter must necessarily be made up of only finitely many intervals. With a point of zero density, it is possible to have a region of finite mass consisting of infinitely many interval that are accumulating at the point of zero density.
\begin{definition}
A region $R$ with mass $M$ and perimeter $P$ is said to be \textbf{isoperimetric} (and is also referred to as a \textbf{bubble} or a \textbf{1-bubble}) if, out of all possible regions with mass $M$, $R$ has the least perimeter.
\end{definition}

%
%
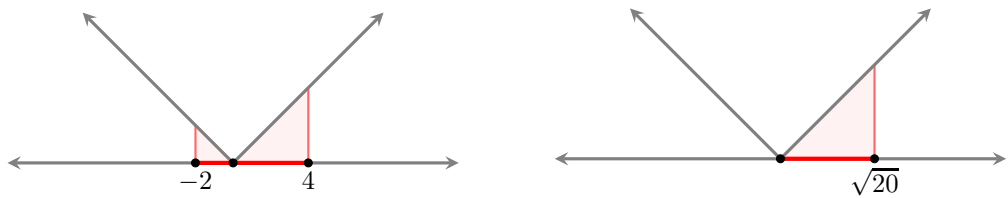
\begin{figure}[htbp]  
    \centering 
    \begin{subfigure}[h]{0.4\textwidth}
    \begin{tikzpicture} 
    \filldraw[color=red!60, fill=red!5, thick] (0,0) -- (-.5,0) -- (-.5,.5) -- cycle;
    \filldraw[color=red!60, fill=red!5, thick] (0,0) -- (1,0) -- (1,1) -- cycle;
    
    \draw[gray,stealth-,very thick] (-2,2) -- (0,0);
    \draw[gray,-stealth,very thick] (0,0) -- (2,2);
    \draw[gray,stealth-stealth,very thick] (-3,0) -- (3,0);
    
    \filldraw[red, ultra thick] (0,0) -- (1,0);
    \filldraw[red, ultra thick] (-.5,0) -- (0,0);

    \filldraw (0,0) circle (1.5pt);
    \filldraw (1,0) circle (1.5pt) node[anchor=north] {$4$};
    \filldraw (-.5,0) circle (1.5pt)node[anchor=north] {$-2$};
    
    \end{tikzpicture} 

    \end{subfigure}
    \hfill
    \begin{subfigure}[h]{0.4\textwidth}
    \begin{tikzpicture} 
    \filldraw[color=red!60, fill=red!5, thick] (0,0) -- (1.25,0) -- (1.25,1.25) -- cycle;
    
    \draw[gray,stealth-,very thick] (-2,2) -- (0,0);
    \draw[gray,-stealth,very thick] (0,0) -- (2,2);
    \draw[gray,stealth-stealth,very thick] (-3,0) -- (3,0);
    
    \filldraw[red, ultra thick] (0,0) -- (1.25,0);

    \filldraw (0,0) circle (1.5pt);
    \filldraw (1.25,0) circle (1.5pt) node[anchor=north] {$\sqrt{20}$};
    \end{tikzpicture}
    \end{subfigure}
        \caption{Above, we see an example with density function $f(x) = |x|$, and regions of weighted mass 10 arranged in two different configurations. The first configuration, defined on the interval $[-2,4]$, has weighted perimeter of 6. The second configuration, defined on the interval $[0,\sqrt{20}]$, has weighted perimeter of $\sqrt{20} \approx 4.47$. This second configuration is the isoperimetric 1-bubble, as shown in \cite{HuangMorgan19}. Note that in each picture, the ``region'' is an interval, while the weighted volume is calculated as (and displayed as) area under the density function.}
        \label{Figure1Example}
    \end{figure}

We can also consider multiple regions on the same number line. In such a situation, each region consists of (possibly multiple, disjoint) intervals. Two intervals, each from a different region, are either completely disjoint or meet at a single endpoint. In addition to measuring the mass and perimeter of each region, we can measure the \textbf{total (weighted) perimeter} by summing the perimeter of each of the regions. In doing so, any point where two intervals meet is only counted once. 

\begin{definition}
A configuration of n regions $R_1, \dots, R_n$ with masses $M_1, \dots, M_n$ is said to be \textbf{isoperimetric} (and is referred to as an \textbf{$n$-bubble}) if, out of all possible configurations of regions with the same masses, this configuration has the least total perimeter.
\end{definition}

To begin exploring isoperimetry  in earnest, we can start by looking at one side of the number line (say, without loss of generality, the positive side $x \geq 0$). Our first result says that we can benefit (resulting in lower perimeter) by consolidating our regions into adjacent intervals near the origin.

\begin{proposition}
On the positive real number line $(x \geq 0)$ with an increasing density function $f$, suppose we have a number of regions $R_1... R_n$ with masses $M_1, \dots , M_n$. Note that each region could consist of one or more intervals, and that there might be ``empty'' intervals that do not correspond to any particular region. Nevertheless, we can create a new configuration of regions which simultaneously preserve mass and decrease perimeter. In this new configuration, each region consists of a single interval, and there are no empty intervals.
\label{PropCondensed}
\end{proposition}

\begin{proof}
Suppose we have an arbitrary configuration on $n$ regions, each consisting of (possibly multiple) intervals on the non-negative axis $x \geq 0$. WLOG, we identify $R_1$ to have a maximum value (and rightmost endpoint) of $b_1$, $R_2$ to have a maximum value of $b_2$, and so on, named so that $b_1 < b_2 < \dots < b_n$. As seen in Figure \ref{Figure3Consolidate}, we can create a new configuration in which regions consist of a single interval, each of which are adjacent and ordered so that region $R_i$ sits to the left of region $R_j$ if $i < j$. Furthermore, this can clearly be done in such a manner that the mass of each region is preserved.

%
%
     \begin{figure}[H]
    \centering
    \begin{subfigure}[h]{0.4\textwidth}    
    \begin{tikzpicture}
    \draw[gray, -stealth,very thick](0,0) -- (3,3);
    \draw[gray, -stealth,very thick] (0,0) -- (4,0);

    \filldraw[green,ultra thick] (.9,0) -- (1.45,0);

    \filldraw[blue,ultra thick] (.3,0) -- (.55,0);
    \filldraw[blue,ultra thick] (3,0) -- (3.35,0);

    \filldraw[red,ultra thick] (2.15,0) -- (1.8,0);
    \filldraw[red,ultra thick] (2.45,0) -- (3,0);


    \filldraw (.9,0) circle (1pt);
    \filldraw (1.45,0) circle (1pt) node[anchor=north] {\tiny $b_1$};
    \filldraw (.3,0) circle (1pt) ;
    \filldraw (.55,0) circle (1pt);
    \filldraw (2.15,0) circle (1pt);
    \filldraw (1.8,0) circle (1pt);
    \filldraw (3,0) circle (1pt) node[anchor=north] {\tiny $b_2$};
    \filldraw (3.35,0) circle (1pt) node[anchor=north] {\tiny $b_3$};
    \filldraw (2.45,0) circle (1pt);

    \end{tikzpicture}
	\end{subfigure}
 \hfill
     \begin{subfigure}[h]{0.4\textwidth}    
    \begin{tikzpicture}
    \draw[gray, -stealth,very thick](0,0) -- (3,3);
    \draw[gray, -stealth,very thick] (0,0) -- (4,0);

    \filldraw[green,ultra thick] (0,0) -- (0.8,0);

    \filldraw[blue,ultra thick] (1.8,0) -- (2.5,0);

    \filldraw[red,ultra thick] (0.8,0) -- (1.8,0);


	\filldraw (0,0) circle (1pt);    
    \filldraw (0.8,0) circle (1pt) node[anchor=north] {\tiny $b_1'$};
    \filldraw (1.8,0) circle (1pt) node[anchor=north] {\tiny $b_2'$};
    \filldraw (2.5,0) circle (1pt) node[anchor=north] {\tiny $b_3'$};

    \end{tikzpicture}
    \end{subfigure}
    \caption{We see how an arbirtrary arrangement (on the left) can be consolidated (shown the right) in a manner that lowers perimeter.}
	\label{Figure3Consolidate}
\end{figure}
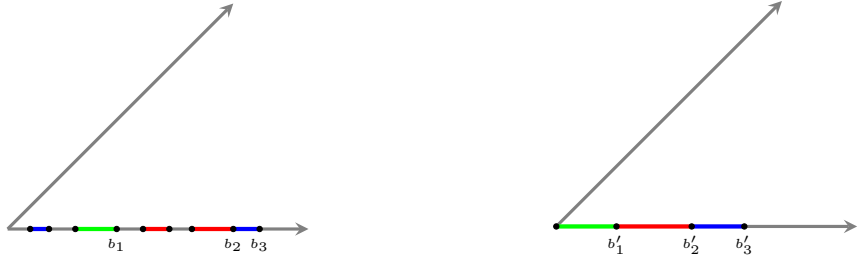

Because of how we have constructed these regions, it is clear that the new rightmost endpoint $b_j'$ of the region $R_j$ satisfies $0 < b_j' \leq b_j$. Since $f$ is assumed to be increasing, we know $f(b_j') \leq f(b_j)$. This gives us the following inequalities for the new and original perimeter:
$$
P_{new} = \displaystyle \sum_{j=1}^n f(b_j') \leq \displaystyle \sum_{j=1}^n f(b_j) \leq P_{orig} 
$$
completing the proof.
\end{proof}

\begin{corollary}
If a configuration of $n$ regions is isoperimetric, then the configuration consists of at most $2n$ adjacent intervals, with the origin contained in at least one of the intervals (either at an interior or an endpoint).
\end{corollary}

\begin{proof}
Our previous proposition says that we can reconfigure our regions so that there are at most $n$ intervals on the positive side, and at most $n$ intervals on the negative side. (i.e., each region will have at most one interval on each side of the origin.) Furthermore, as the consolidation from before occurs at the origin, it is clear that the origin will either be between two different regions (in which case it's an endpoint of each), or in a single region (with some mass to be found on both the positive and negative side).
\end{proof}

Beyond simply knowing that our regions accumulate near the origin, we also learn that we can reorder them so that they ``grow'' from smallest to largest mass as we move further out from the origin. This follows immediately from iterating the transposition lemma, which we state below.

\begin{proposition}[Transposition Lemma]
Consider a configuration of regions $R_i$ ($i = 1,2$) such that the regions are individual intervals adjacent to each other on the positive $x$-axis. Suppose each $R_i$ has mass $M_i$, and suppose $M_1 < M_2$. Then weighted perimeter is minimized when the interval for $R_1$ lies to the left of the interval for $R_2$.
\end{proposition}

\begin{proof}
This is seen immediately in Figure \ref{Figure4Transpose}. Keeping the leftmost and rightmost endpoints fixed, and transposing the two intervals if necessary, we see that the perimeter is lowered when the interval with less mass is placed to the left (as it moves the internal endpoint to the left).
\end{proof}

%
%

     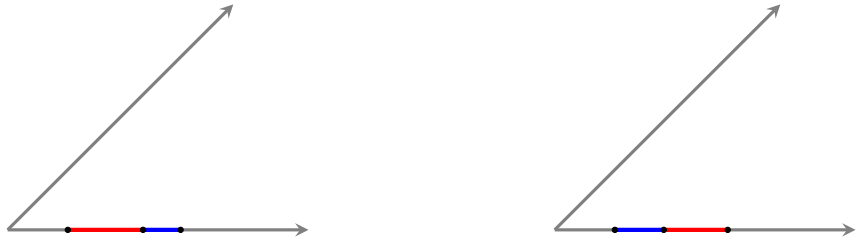
\begin{figure}[H]
    \centering 
    \begin{subfigure}[h]{0.4\textwidth}    
    
    \begin{tikzpicture}
    \draw[gray, -stealth,very thick](0,0) -- (3,3);
    \draw[gray, -stealth,very thick] (0,0) -- (4,0);

    \filldraw[blue,ultra thick] (1.8,0) -- (2.3,0);

    \filldraw[red,ultra thick] (0.8,0) -- (1.8,0);


    \filldraw (0.8,0) circle (1pt);
    \filldraw (1.8,0) circle (1pt);
    \filldraw (2.3,0) circle (1pt);

    \end{tikzpicture}
	\end{subfigure}
 \hfill
     \begin{subfigure}[h]{0.4\textwidth}    
    
    \begin{tikzpicture}
    \draw[gray, -stealth,very thick](0,0) -- (3,3);
    \draw[gray, -stealth,very thick] (0,0) -- (4,0);

    \filldraw[red,ultra thick] (1.45,0) -- (2.3,0);

    \filldraw[blue,ultra thick] (0.8,0) -- (1.45,0);


    \filldraw (0.8,0) circle (1pt);
    \filldraw (1.45,0) circle (1pt);
    \filldraw (2.3,0) circle (1pt);

    \end{tikzpicture}
	\end{subfigure}
 
    \caption{Transposition of adjacent intervals: If we transpose $R_1$ and $R_2$ without changing their global location, the lower perimeter results from placing the smaller interval closer to the origin (as the only endpoint to move is the internal endpoint). Iterating the transposition lemma guarantees that, on one side of the interval, regions will be ordered (according to weighted area) from smallest to largest as we move away from the origin.}
    \label{Figure4Transpose}
\end{figure}

\begin{definition}
A configuration of $n$ regions is said to be in a \textbf{condensed configuration} if it satisfies the following conditions:
	\begin{itemize}
	\item the configuration consists of at most $2n$ adjacent intervals, with each region contributing at most one interval on each side of the origin.
	\item the origin is contained in at least one of the intervals, possibly as an endpoint.
	\item The intervals are all adjacent, and -- working from the origin outward -- they increase in size.
	\end{itemize}
	\label{DefCondensed}
\end{definition}

\begin{remark}
Due to Proposition \ref{PropCondensed}, we know a necessary condition for a configuration to be isoperimetric is for it to be condensed.
\end{remark}

We conclude this section with an important comparison lemma that allows us to compare intervals of appropriate masses. Although it is stated for the positive real number line, it is clear (due to the symmetry of $f$) that there is a corresponding statement on the negative real number line.

\begin{proposition}
On the positive real number line with a symmetric, radially incresaing density $f$: Consider two intervals $R_1 = [a_1, b_1]$,  $R_2 = [a_2, b_2]$. Suppose that each of these intervals contains the same mass ($M_1 = M_2$), and suppose the innermost endpoint of the first region is closer to the origin than the second $(a_1 < a_2)$. Then the following two inequalities hold true: 
$$
  b_1 < b_2 \qquad \text{and} \qquad b_1 - a_1 > b_2 - a_2.
$$
\label{PropInequality1}
\end{proposition}

\begin{proof}
The first inequality is an immediate consequence of the two intervals enclosing the same mass. For the final inequality, note that since  $[a_2, b_2]$ is further away from the origin. Since our density function is radially increasing as we move away from the origin, the  average density of $[a_2, b_2]$ will be larger than that of $[a_1, b_1]$, and since they each contain the same amount of mass, this means the interval $[a_2, b_2]$ will be smaller.
\end{proof}

\begin{corollary}
The first inequality above continues to hold true in the cases where $M_2 \geq M_1$. The second inequality also holds true in the cases where $M_2 \leq M_1$.
\end{corollary}

\section{The First Variation of Mass and Perimeter}

In this section, we briefly introduce the concept of continuously varying configurations and the first variation formula. We will introduce this for an arbitrary density function $f$ over the real number line.

\subsection{The First Variation Formulas}

Consider an interval of $[a,x]$. We have already seen how the weighted mass $M$ and perimeter $P$ are measured as $\int_a^x f$ and $f(a) + f(x)$, respectively. An immediate consequence of this is that

\begin{align}
    \frac{dM}{dx} &= f(x)\\
    \frac{dP}{dx} &= f'(x).
\end{align}
Suppose that we let $x = x(t)$ vary as a function of time, and with a specified velocity $x'(t)$. This will allow us to calculate 

\begin{align}
    \frac{dM}{dt} &= f(x(t)) x'(t)\\
    \frac{dP}{dt} &= f'(x(t)) x'(t).
\end{align}
If we choose our velocity to equal $1/f(x(t))$, these equations become

\begin{align}
    \frac{dM}{dt} &= 1\\
    \frac{dP}{dt} &= f'(x(t)) / f(x(t)) = \frac{d}{dx}\left[ \log(f(x(t))) \right].
\end{align}

%
%
\begin{figure}[H]
    \centering
    \begin{tikzpicture}
    \draw[gray, stealth-,very thick] (-3,3) -- (0,0);
    \draw[gray, -stealth,very thick](0,0) -- (3,3);
    \draw[gray, stealth-stealth,very thick] (-4,0) -- (4,0); 
    
    
    \filldraw[blue,ultra thick] (.7,0) -- (1.7,0);


    \draw[red, -stealth,very thick] (1.7,0) -- (2.5,0);

    \filldraw (.7,0) circle (2pt) node[anchor=north] {$a$};
    \filldraw (1.7,0) circle (2pt) node[anchor=north] {$x$};


    \end{tikzpicture}
    \caption{The interval $[a,x(t)]$ has one endpoint dragged to the right at speed $x'(t) = 1/f(x)$.}
    \label{fig:initPositions}
\end{figure}
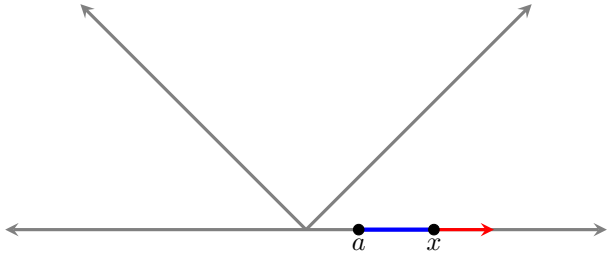

An immediate consequence of this is as follows:

\begin{proposition}[The first variation formula]
Suppose we take an interval $[a,b]$, and drag \emph{both} endpoints to the right at speed $1/f(x)$. Then the weighted mass of the interval will not change, while the perimeter's instantaneous rate of change will be $(\log(f))'(a) + (\log(f))'(b)$
\label{Prop1VF}
\end{proposition}

\begin{corollary}
Suppose we taken an isoperimetric $n$-bubble, comprised of a number of intervals with all endpoints of all intervals listed as $x_1, x_2, \dots, x_m$. Then 
$$
\sum_{i=1}^m \log(f)'(x_i) = 0.
$$
\end{corollary}

\begin{proof} If our region is isoperimetric, it has minimal perimeter out of all other configurations that have regions with the same prescribed masses. This means, as we vary our configuration by dragging every endpoint to the right at speed $1/f$, our masses will stay the same and our perimeter must have been at a local/global minimum. Thus, its instantaneous rate of change should be 0.
\end{proof}

\begin{remark}
With our chosen density function $f(x)$, the first variation formula is not defined when there is an endpoint at $x=0$. However, as long as we are not dragging an endpoint across the origin, we can use the first variation formula as usual. Indeed, we will see that the point of zero density at the origin leads endpoints of regions to naturally ``slide'' towards the origin.
\end{remark}

\subsection{Consolidating regions using the first variation}

We have already seen that an $n$-bubble will consist of at most $2n$ intervals. We use the first variation formula in a number of ways to reduce the number of total intervals. We begin by identifying the different ways in which multiple intervals can be relatively positioned. 

\begin{definition}
Suppose we have a condensed configuration of $n$ regions, and suppose that there are  two regions (call them $A$ and $B$) each with two intervals, one on each side of the origin. Call these intervals $A^-$, $A^+$, $B^-$, and $B^+$  depending on their position relative to the origin. Then the regions $A$, $B$ are said to be  \textbf{alternating} if their relative positions, ordered left to right, are as $A^-$, $B^-$, $A^+$, $B^+$ or as $B^-$, $A^-$, $B^+$, $A^+$.

If $A$ and $B$ are not in an alternating pattern, this means that the relative positions are either as $A^-$, $B^-$, $B^+$, $A^+$ or as $B^-$, $A^-$, $A^+$, $B^+$. In this case, we say that $A$, $B$ are \textbf{nested}.
\end{definition}

Our next lemma will guarantee that alternating patterns cannot be isoperimetric: if such a configuration exists, we can use the first variation formula to create a continuous movement that preserves the masses of our $n$ regions while reducing total perimeter. The core idea present -- simultaneous siphoning of area from each of the inner intervals in the alternating pattern -- was used in \cite{ChambersBongiovanniETAL18} to reduce the number of possible intervals in a 2-bubble. 

\begin{proposition}
On the real number line with a radially increasing density function $f$: Suppose there exists a condensed configuration of $n$ regions, with two regions in an alternating pattern. Then we can create a new configuration the same $n$ masses and lower total perimeter that eliminates at least one of the inner intervals in the alternating pattern. 
\end{proposition}

\begin{proof}
Suppose we have a condensed configuration of $n$ regions such that two regions (identified as $A$ and $B$) exist in an alternating pattern. 
WLOG, suppose points in these intervals satisfy $A^- \leq B^- \leq 0 \leq A^+ \leq B^+$.

Take the right endpoint of $A^-$, the left endpoint of $B^-$, and all interval endpoints between these two. We drag these endpoints to the right at speed $1/f(x)$. Simultaneously, we take the right endpoint of $A^+$, the left endpoint of $B^+$, and all interval endpoints between these two, and drag these endpoints to the left at speed $1/f(x)$. This process is illustrated in Figure \ref{Figure6Alternating}. According to Proposition \ref{Prop1VF}, this process does not change the weighted masses of any region for which both endpoints are moving. Additionally, the masses for $A$ and $B$ are not changing: the mass gained by $A^-$ (and $B^+$) is identical to the mass lost by $A^+$ (and $B^-$). Since $f$ is radially increasing and our endpoints are all moving towards the origin, this variation reduces total weighted perimeter. Since no moving endpoint is in danger of reaching the origin, this process can continue until either $B^-$ or $A^+$ has shrunk entirely to size 0. We see that perimeter was made to decrease throughout this process, and continued to decrease until one of the central intervals in the alternating pattern completely disappears.
\end{proof}

%
%

\begin{figure}[H]
    \centering
    \begin{subfigure}[h]{0.3\textwidth}    
    \begin{tikzpicture}[scale=0.4]
    \draw[gray, stealth-,thick](-3,3) -- (0,0);
    \draw[gray, -stealth,thick](0,0) -- (3,3);
    \draw[gray, stealth-stealth,thick](-4.5,0) -- (4.5,0); 
    \filldraw[green,ultra thick] (-.75,0) -- (-2,0);
    \filldraw[green,ultra thick] (3.1,0) -- (3.6,0);

	\node[anchor=north] at (-1.4,0) {\tiny $G$};
	\node[anchor=north] at (3.35,0) {\tiny $G$};
	
    \filldraw[blue, ultra thick] (-3,0) -- (-2, 0);
    \filldraw[blue,ultra thick] (0,0) -- (.95,0);
    
	\node[anchor=north] at (-2.5,0) {\tiny $B$};
	\node[anchor=north] at (0.5,0) {\tiny $B$};
    
    \filldraw[brown,ultra thick] (0,0) -- (-.75,0);
    \filldraw[brown,ultra thick] (.95,0) -- (1.65,0);
    \filldraw[yellow,ultra thick] (1.65,0) -- (3.1,0);
    
    \filldraw (-.75,0) circle (2pt);
    \filldraw (-2,0) circle (2pt);
    \filldraw (-3,0) circle (2pt);
    \filldraw (0,0) circle (2pt);
    \filldraw (0.95,0) circle (2pt);
    \filldraw (1.65,0) circle (2pt);
    \filldraw (3.1,0) circle (2pt);
    \filldraw (3.6,0) circle (2pt);

    \end{tikzpicture}

	\end{subfigure}
 	\hfill
    \begin{subfigure}[h]{0.3\textwidth}    
    \begin{tikzpicture}[scale=0.4]
    \draw[gray, stealth-,thick](-3,3) -- (0,0);
    \draw[gray, -stealth,thick](0,0) -- (3,3);
    \draw[gray, stealth-stealth,thick](-4.5,0) -- (4.5,0); 
    \filldraw[green,ultra thick] (-.75,0) -- (-2,0);
    \filldraw[green,ultra thick] (3.1,0) -- (3.6,0);
    \filldraw[blue, ultra thick] (-3,0) -- (-2, 0);
    \filldraw[blue,ultra thick] (0,0) -- (.95,0);
    \filldraw[brown,ultra thick] (0,0) -- (-.75,0);
    \filldraw[brown,ultra thick] (.95,0) -- (1.65,0);
    \filldraw[yellow,ultra thick] (1.65,0) -- (3.1,0);

	\node[anchor=north] at (-2.5,0) {\tiny $B$};
	\node[anchor=north] at (0.5,0) {\tiny $B$};

	\node[anchor=north] at (-1.4,0) {\tiny $G$};
	\node[anchor=north] at (3.35,0) {\tiny $G$};    
    
    \draw[red, -stealth,very thick] (-2,0) -- (-1.7,0);
    \draw[red, stealth-,very thick] (2.8,0) -- (3.1,0);
    \draw[red, stealth-,very thick] (1.35, 0) -- (1.65,0);
    \draw[red, stealth-,very thick] (0.65, 0) -- (0.95,0);
    
    \filldraw (-.75,0) circle (2pt);
    \filldraw (-2,0) circle (2pt);
    \filldraw (-3,0) circle (2pt);
    \filldraw (0,0) circle (2pt);
    \filldraw (0.95,0) circle (2pt);
    \filldraw (1.65,0) circle (2pt);
    \filldraw (3.1,0) circle (2pt);
    \filldraw (3.6,0) circle (2pt);

    \end{tikzpicture}
	\end{subfigure}
 	\hfill
    \begin{subfigure}[h]{0.3\textwidth}
    \begin{tikzpicture}[scale=0.4]
    \draw[gray, stealth-,thick](-3,3) -- (0,0);
    \draw[gray, -stealth,thick](0,0) -- (3,3);
    \draw[gray, stealth-stealth,thick](-4.5,0) -- (4.5,0); 
    \filldraw[green,ultra thick] (-.75,0) -- (-1.3,0);
    \filldraw[green,ultra thick] (2.0,0) -- (3.6,0);
    \filldraw[blue, ultra thick] (-3,0) -- (-1.3, 0);
    \filldraw[brown,ultra thick] (0,0) -- (-.75,0);
    \filldraw[brown,ultra thick] (0,0) -- (0.8,0);
    \filldraw[yellow,ultra thick] (0.8,0) -- (2.0,0);

	\node[anchor=north] at (-2.2,0) {\tiny $B$};

	\node[anchor=north] at (-1,0) {\tiny $G$};
	\node[anchor=north] at (2.8,0) {\tiny $G$};
    
    \filldraw (-.75,0) circle (2pt);
    \filldraw (-1.3,0) circle (2pt);
    \filldraw (-3,0) circle (2pt);
    \filldraw (0,0) circle (2pt);
    \filldraw (0.8,0) circle (2pt);
    \filldraw (2,0) circle (2pt);
    \filldraw (3.6,0) circle (2pt);

    \end{tikzpicture}
    \end{subfigure}
    \caption{In the first image, the blue region ``B'' and green region ``G'' consist of alternating intervals. In the second image, we see which endpoints we slide to reduce perimeter and maintain constant weighted volumes. In the final image, the blue region closest to the origin was fully moved to the other blue mass. At this point, the sliding is done and the alternating pattern has been eliminated.}
    \label{Figure6Alternating}
\end{figure}
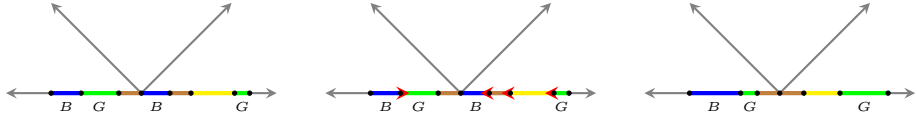

Note that the proposition above holds for all density functions that are radially increasing, because every perimeter point that is moving is moving towards the origin. Our next argument allows us to eliminate any 2-interval region, combining it into a single interval on one side of the origin, in the situation where our density function is log-concave.

\begin{proposition}
Suppose we have $n$ regions in a condensed form on the number line, and suppose that one region can be written as the union of two intervals, one non-positive and one non-negative. Then there exists a configuration $n$ regions with the same masses as before, but with lower perimeter than before, achieved by combining the region identified above into a single interval on one side of the origin.
\end{proposition}

\begin{proof}
Suppose we have our condensed configuration of regions. Identify region $R = R^- \cup R^+ = [r_1, r_2] \cup [r_3, r_4]$ as split into intervals on the negative and positive side of the origin, respectively. 
Then we proceed as follows: first, take all endpoints less than $r_2$, as well as all endpoints greater than $r_3$. (Note that this set of endpoints includes $r_1$ and $r_4$.) Then, shift \emph{only these endpoints} to the right at speed $1/f$. According to the first variation formula for enclosed mass, this results in movement that keeps the weighted mass of each region constant (with $R$'s mass staying constant because $[r_1, r_2]$ is losing mass at the same rate $[r_3, r_4]$ is gaining mass). And by the first variation formula for perimeter, we know the total perimeter changes at rates of

\begin{align*}
\frac{dP}{dt} &= \sum f'(x(t)) / f(x(t))\\
\frac{d^2P}{dt^2} &= \sum \frac{f(x(t)) f''(x(t)) - \left[f'(x(t))\right]^2}{\left[f(x(t))\right]^3}
\end{align*}
where the summation is taken over all perimeter points that are moving. 
Notably, we can observe that the second derivative of perimeter is negative if $f$ is log-concave. So the perimeter function $P$, when dragging every perimeter point to the right, is concave down. There are therefore two possibilities:
	\begin{enumerate}
	\item If perimeter is decreasing (so $P' < 0$), it will continue to decrease. We can therefore continue in this manner until the endpoint $r_1$ has been dragged to collide with endpoint $r_2$.
	\item If perimeter is increasing, we instead drag every one of our targeted points to the \emph{left} with speed $1/f(x)$. This will keep the masses the same, will change the sign of $P'$, and will leave $P''$ negative. The result is that perimeter will decrease, and will continue to decrease until the endpoint $r_4$ collides with the endpoint $r_3$.
	\end{enumerate}

Thus, we see that we can drag endpoints in a manner that decreases perimeter until either $[r_1, r_2]$ or $[r_3, r_4]$ disappears. Either way, we are left with the entirety of region $R$ on one side of the origin.
\end{proof}

\begin{corollary}
If our density function $f$ is radially symmetric and log-concave: Suppose we have $n$ regions in a condensed form on the number line, and suppose the origin is contained in the interior of one region. Then our configuration is not isoperimetric, and we can lower perimeter by moving this region entirely to one side of the origin.
\end{corollary}

\begin{proof}
The same argument as above applies, with the region in question being interpreted as the union of two intervals, each directly adjacent to the origin. In this interpretation, every endpoint (except the origin) will be dragged to the right or the left in the variation.
\end{proof}

\begin{corollary}
If our density function $f$ is radially symmetric and log-concave: An isoperimetric configuration with $n$ regions consists of exactly $n$ intervals, with one interval per region. The origin will sit as an endpoint to two distinct regions.
\label{CorOneIntervalPerRegion}
\end{corollary}

We conclude this section with an important inequality relating perimeter points to intervals that contain the same mass. The result, due to the log-concavity of the density function, should be viewed as similar to those of Proposition \ref{PropInequality1}.

\begin{proposition}

On the positive real number line with a log-concave density function $f$: Consider two intervals $R_1 = [a_1, b_1]$,  $R_2 = [a_2, b_2]$. Suppose that each of these intervals contains the same mass ($M_1 = M_2$), and suppose the innermost endpoint of the first region is closer to the origin than the second $(a_1 < a_2)$. Then $$f(b_1) - f(a_1) \geq f(b_2) - f(a_2).$$
\label{PropInequality2}
\end{proposition}

\begin{proof}
Let $Q(t) = f(b(t)) - f(a(t))$ with $a(0) = a_1$ and $b(0) = b_1$, and consider what happens to $Q(t)$ as we drag points $a,b$ to the right at speed $1/f$. The first variation formula tells us 
$$
Q'(t) = \frac{f'(b(t))}{f(b(t))} - \frac{f'(a(t))}{f(a(t))}.
$$
Since $f$ is log-concave, we know the quantity $f'/f$ is \emph{decreasing}, meaning that (as long as $a < b$) we will have $Q'(t) \leq 0$.  Therefore, $Q(t)$ is a monotone decreasing quantity, and will be smaller by the time it gets to points $a_2$, $b_2$. This completes the proof.

\end{proof}

\subsection{A standard $n$-bubble configuration \label{Section33}}

We are now in a position to identify our ``standard'' $n$-bubble configuration. Corollary \ref{CorOneIntervalPerRegion} has shown that every $n$-bubble must have exactly $n$ intervals positioned along the real number line; that the intervals must be arranged in a condensed manner (in the spirit of Definition \ref{DefCondensed}); and that the origin will appear as an endpoint for two distinct intervals/regions. The final question is the appropriate order for the intervals, i.e. on which side of the origin they will appear.

\begin{definition}
Suppose we have a set of $n$ fixed masses $\{ M_1, \dots, M_n \}$ that satisfy  $ 0 < M_1 \leq M_2 \leq \dots \leq M_n$. Then the \textbf{standard configuration} of these masses will be an arrangement of $n$ regions (with $R_i$ containing mass $M_i$), such that
	\begin{itemize}
	\item The regions are condensed (in the spirit of Definition \ref{DefCondensed}) and appear as single intervals (in the spirit of Corollary \ref{CorOneIntervalPerRegion}).  
	\item The region $R_i$ is found on the positive side or the negative side of the origin depending solely on the parity of its index, so odd-indexed regions appear the opposite side of the even-indexed regions. 
	\end{itemize}
\end{definition}

Because of the nature of condensed regions (following Definition \ref{DefCondensed} and the transposition lemma), we know that smaller regions will appear closer to the origin. This, along with the even/odd index split across the origin, completely determines the positions of the intervals. Figure \ref{Fig2StandardPosition} shows a standard 5-bubble.

%
%
\begin{figure}[H]
    \centering
    \begin{tikzpicture}
    \draw[gray, stealth-,very thick] (-3,3) -- (0,0);
    \draw[gray, -stealth,very thick](0,0) -- (3,3);
    \draw[gray, stealth-stealth,very thick] (-4,0) -- (4,0); 
    
    \filldraw[brown,ultra thick] (0,0) -- (.8,0);
    
    \filldraw[blue,ultra thick] (0,0) -- (-1.1,0);

    \filldraw[pink,ultra thick] (.8,0) -- (1.7,0);

    \filldraw[yellow,ultra thick] (1.7,0) -- (3.3,0);

    \filldraw[red,ultra thick] (-1.1,0) -- (-2.4,0);



    \filldraw (0,0) circle (2pt);
    \filldraw (.8,0) circle (2pt);
    \filldraw (1.7,0) circle (2pt);
    \filldraw (-1.1,0) circle (2pt);
    \filldraw (3.3,0) circle (2pt);
    \filldraw (-2.4,0) circle (2pt);

	\node[anchor=north] at (0.5,0) {\tiny $R_1$};
	\node[anchor=north] at (-0.5,0) {\tiny $R_2$};
	\node[anchor=north] at (1.25,0) {\tiny $R_3$};
	\node[anchor=north] at (-1.8,0) {\tiny $R_4$};
	\node[anchor=north] at (2.45,0) {\tiny $R_5$};


    \end{tikzpicture}
    \caption{A standard $5$-bubble, consisting of regions $R_i$ of weighted mass $M_i$ satisfying $M_1 \leq M_2 \leq M_3 \leq M_4 \leq M_5$. Note that the masses are weighted by the function $f(x) = |x|$, so (e.g.) even though the interval associated with $R_2$ appears larger than the interval associated with $R_3$, the weighted masses satisfy $M_3 > M_2$.}
    \label{Fig2StandardPosition}
\end{figure}
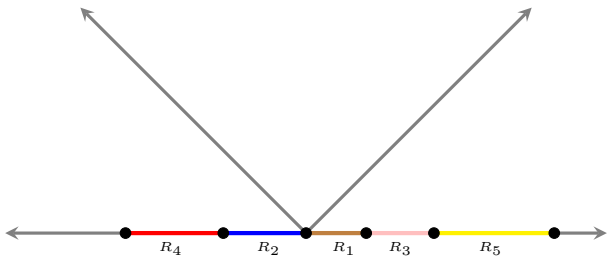

\begin{theorem}
For any $n \in \mathbb{N}$, and any collection of $n$ masses, the standard configuration is the only isoperimetric $n$-bubble with these masses.
\label{TheoremNBubble}
\end{theorem}

\section{Manipulating volumes of regions}

In this section, we explore how changing the weighted volumes of the regions will affect the isoperimetric configuration. These arguments will be used later in our induction/contradiction argument by changing our (non-standard) isoperimetric configuration to a related (non-standard) isoperimetric configuration with different sized regions.


\subsection{Adding a region of small size $\epsilon$}

We want to imagine what happens if we take a set of $n$ regions in standard position and ``add in'' an $(n+1)$th region. A priori, we want this new region to be able to have any mass (so we are not simply creating a new region of largest mass or of smallest mass). We could imagine doing this by injecting the new region into the configuration (possibly as multiple intervals), and shifting/rearranging the rest of the regions to accommodate this new addition. A special case of adding a new region is identified here: 

\begin{definition}
Suppose we have $n$ regions arranged in standard position. Then we say a newly-introduced $(n+1)$th region will\textbf{maintain standard position} if the new configuration with $n+1$ regions is also in standard position. 
\end{definition}

Of course, by simply inserting a new region (even as a single interval) we might lose the ``positive/negative parity'' of the (reindexed) regions. To maintain this parity, we either need our inserted region to be smaller than all other regions (and inserted directly adjacent to the origin); larger than all other regions (and inserted on the outside of the configuration); or, when inserting it into the middle of the configuration, all regions with more mass than the inserted region must be rearranged (by swapping sides across the origin). In what follows below, we see that introducing a mass into a standard configuration of $n$ regions is best done by maintaining standard position, if the introduced mass is small.

\begin{proposition}
Suppose that we have $n$ regions (of mass $M_1 \leq \dots \leq M_n$) arranged in a standard position. Then there exists an $\epsilon$, depending on $n$, such that the following is true: when adding an $(n+1)$th region of size $\leq \epsilon$ to the configuration, perimeter is minimized by adding the new region a way that maintains standard position.
\end{proposition}

\begin{proof}
We observe that our original $n$ regions have a total of $n$ endpoints other than the origin, and each region can be identified by its endpoint furthest from the origin. Call these endpoints $b_i$. Then the total perimeter is $\sum_{i} f(b_i)$. When adding an $(n+1)$th region, we see that adding it in a manner that places it adjacent to the origin will result in each other interval ``shifting'' a little to the right or left, slightly increasing the value of each $b_i$ on that side of the origin. We can observe that a region of mass $\epsilon$ inserted next to the origin (say, on the positive side) will have an endpoint $b_*$ determined $\int_0^{b_*} f = \epsilon$, and a new perimeter value of $f(b_*)$. Furthermore, all endpoints on this side of the origin will shift. Due to Proposition \ref{PropInequality2}, each endpoint's shift will result in a perimeter increase less than $f(b_*)$. So as a (very rough) estimate, inserting our new mass adjacent to the origin will increase total perimeter less than $(n+1)f(b_*)$. Since our density has a point of 0 density at the origin, by making $\epsilon$ small we will make this quantity less than any $f(b_i)$, which is an underestimate for the perimeter added if our new region is inserted adjacent to another $M_i$.

Therefore, our small $(n+1)$th region must be inserted adjacent to the origin to minimize added perimeter. A quick check allows us to verify that it will be better inserted into standard position (which will be across the origin from $M_1$).
\end{proof}

The previous result is not surprising, as it can be viewed through a lens of continuity. Indeed, we largely expect isoperimetric configurations to vary continuously as the masses of the regiouns vary in a continuous manner. In this way, we can imagine the standard configuration of $n+1$ regions varying continuously as the smallest region (noted above as having volume $\epsilon$) shrinks to 0. 

One concern of ours is that there might exist a set of set of $n+1$ masses for which there are two distinct isoperimetric configurations. One could imagine that varying one (or more) of the masses in such a set could result in diverging isoperimetric configurations. 

\begin{definition}
A set of masses $\{M_i \}$ for which two configurations of regions are simultaneously isoperimetric will be called a \textbf{bifurcating set of masses}
\label{DefNonuniqueIsoperimetrc}
\end{definition}

Our result above shows that, when $n$ masses are fixed and one mass is small, our set of masses will not be bifurcating.

\begin{corollary}
Suppose that we know that the standard configuration for $n$ masses is the isoperimetric configuration. Then given $n+1$ masses, where one of the masses is suitably small, we can be guaranteed that the isoperimetric configuration will be the standard position.
\end{corollary}


\begin{corollary}
Suppose it is known that any configuration of $n$ regions has a standard isoperimetric configuration. However, suppose there is a some set of $n+1$ masses whose isoperimetric configuration is non-standard. Then, by shrinking the mass of the smallest region, we can find a bifurcating set of $n+1$ masses. Of the isoperimetric configurations, one will be in standard position. 
\end{corollary}

\begin{proof}
Assume there exists a set of $n+1$ regions (call the regions $R_i$, each of mass $M_i$, with $M_1 \leq \dots \leq M_{n+1}$) whose isoperimetric configuration is non-standard. We know that $M_1$ cannot be too small (or else the configuration would be standard). We will say a mass $x > 0$ is \textbf{below the bifurcation threshold} if the isoperimetric configuration for masses $y, M_2, \dots, M_{n+1}$ is standard for all $y \leq x$. Let $\epsilon$ be defined as $\epsilon = \sup \{ x: x\; \text{is below the threshold} \}$. Then by continuity, we know that the $n+1$ masses $\{ \epsilon, M_2, \dots, M_{n+1} \}$ are bifurcating: there are at least two isoperimetric configurations, one of which is standard.
\end{proof}



\subsection{Inflating a large mass in a bifurcating set of masses}

Here we explore in detail an argument that will be generalized later. Suppose we have a collection of masses $M_1 \leq M_2 \leq \dots \leq M_{n+1}$ that can be configured in two distinct isoperimetrc regions, one standard and one non-standard. Our work in sections 2 and 3 guarantee that in both isoperimetric configurations there are exactly $n+1$ intervals, whose masses increase as they move away from the origin. This means that the region of largest mass must be ``external'' to the configuration in both the standard and non-standard solutions. WLOG, we assume the region of mass $M_{n+1}$ appears on the positive side of the $x$-axis in both images. Let the $n+1$ endpoints (other than the origin) in the standard configuration be denoted as $a_{i}$, and let the endpoints in the non-standard configuration be denoted as $b_i$. Since both configurations are isoperimetric, we have
$$
\sum_{i=1}^{n+1} f(a_i) = \sum_{i=1}^{n+1} f(b_i).
$$
Let us assume for a moment that any configuration of $n$ regions is isoperimetric if and only if it is in standard position. Then we know
$$
\sum_{i=1}^{n} f(a_i) < \sum_{i=1}^{n} f(b_i).
$$
This tells us that $a_{n+1} > b_{n+1}$. An immediate consequence of this follows:

\begin{corollary}
Suppose we it is known that the sole isoperimetric solution with $n$ regions is the standard one, and suppose we have a bifurcating set of masses $\{M_i\}_{i=1}^{n+1}$. WLOG, assume the largest region appears on the positive side in both configurations, and name the endpoints $a_i$ and $b_i$ for the standard and nonstandard solutions respectively. Call the leftmost endpoint of the leftmost region $a_*$ (for the standard configuration) and $b_*$ (for the non-standard). Then $a_* > b_*$.
\end{corollary}

\begin{proof}
From our work above, we know that $a_{n+1} > b_{n+1}$. These are the rightmost endpoints of each of our isoperimetric solutions. Since both configurations have the same total mass between the leftmost and rightmost endpoints (namely, $\sum M_i$), we must have $a_* > b_*$ as well.
\end{proof}

\begin{definition}
Suppose we have two isoperimetric solutions (one standard and one non-standard) with masses $M_1 \leq \dots \leq M_{n+1}$, with $a_i$ and $b_i$ defined as above (for the standard and non-standard solution, respectively). If $b_* < a_* \leq 0 < b_{n+1} < a_{n+1}$, we say the nonstandard solution is \textbf{shifted to the left} of the standard solution.
\end{definition}


Our main result of this section is as follows:

\begin{proposition}
Suppose we have a bifurcating set of masses with two isoperimetric solutions (one standard and one nonstandard). WLOG, assume the largest region is on the positive side of the origin. Suppose the second-largest region, $M_\star$, is on the negative side of the origin. Then, by increasing the value of $M_*$ (and thus inflating the left-most region of each configuration), we can create a new set of masses in which the standard configuration fails to be isoperimetric.
\label{PropOuterMassGrowth}
\end{proposition}

\begin{proof}
Referring to Figure \ref{Figure7MassGrow}, suppose we have two isoperimetric configurations, one standard (with endpoints $a_i$) and one non-standard (with endpoints $b_i$). Enlarging the second-largest mass, $M_*$ will shift $a_*$ and $b_*$ to new values, $a'_*$ and $b'_*$. However, we know the non-standard configuration is shifted to the left, and due to Proposition \ref{PropInequality2} we can conclude $f(a'_*) - f(a_*) > f(b'_*) - f(b_*)$. Calculating the perimeter of our new configuration, we see that the $\sum f(a_i) > \sum f(b_i)$, and that the nonstandard configuration now has less perimeter than the standard. Therefore, the standard configuration with this new set of masses cannot be isoperimetric. \end{proof}

%
%
\begin{figure}[H]
    \centering
    \begin{subfigure}[h]{0.4\textwidth}
    \begin{tikzpicture}[scale=0.7]
    \draw[gray, stealth-,very thick] (-3,3) -- (0,0);
    \draw[gray, -stealth,very thick](0,0) -- (3,3);
    \draw[gray, stealth-stealth,very thick] (-4,0) -- (4,0); 
    
    
    \filldraw[blue,ultra thick] (-1.5,0) -- (2.0,0);
    \filldraw[green,ultra thick] (-2.0,0) -- (-1.5,0);

    \draw[red, -stealth,very thick] (-2,0) -- (-2.5,0);

   	\filldraw (-2.0,0) circle (2pt) node[anchor=north] {$a_*$};
    \filldraw (2.0,0) circle (2pt);
	\filldraw (-1.5,0) circle (2pt);

    \end{tikzpicture}
    	\end{subfigure}
 	\hfill
    \begin{subfigure}[h]{0.40\textwidth}    
        
  \begin{tikzpicture}[scale=0.7]
    \draw[gray, stealth-,very thick] (-3,3) -- (0,0);
    \draw[gray, -stealth,very thick](0,0) -- (3,3);
    \draw[gray, stealth-stealth,very thick] (-4,0) -- (4,0); 
    
    
    \filldraw[blue,ultra thick] (-1.7,0) -- (1.8,0);
    \filldraw[green,ultra thick] (-2.2,0) -- (-1.7,0);

    \draw[red, -stealth,very thick] (-2.2,0) -- (-2.65,0);

   	\filldraw (-2.2,0) circle (2pt) node[anchor=north] {$b_*$};
    \filldraw (1.8,0) circle (2pt);
	\filldraw (-1.7,0) circle (2pt); 

    \end{tikzpicture}    
    
    	\end{subfigure}
    
    \caption{Our two images represent two different isoperimetric configurations (whose details, except for the leftmost region $M_*$, are hidden in the blue region of the $x$-axis). We note that the second picture is left-shifted, so $b_* < a_*$. Inflating the mass $M_*$ will shift these points to the left, but (due to the left-shifted nature) will increase the perimeter for the standard configuration more than the nonstandard.}
    \label{Figure7MassGrow}
\end{figure}
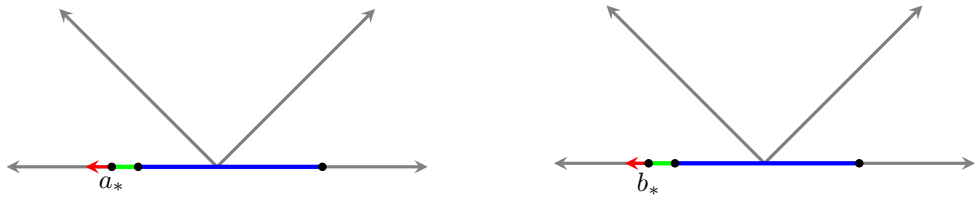

\section{Proof of the $n$-bubble theorem}

In this section, we prove Theorem \ref{TheoremNBubble}'s claim that the only isoperimetric solution is the standard one. We do this by induction. Suppose it is known that every set of $n$ positive masses has the standard configuration as the only isoperimetric solution. For contradiction, assume this is not true for $n+1$. This means there exists at least one configuration of masses $M_1 \leq \dots \leq M_{n+1}$ for which there is a nonstandard isoperimetric solution. Corollary \ref{CorOneIntervalPerRegion} tells us that each region in this nonstandard configuration consists of a single interval, that each interval is entirely on one side of the origin, and that (on its respective side of the interval) the intervals are ordered according to mass size, with the smaller masses closer to the origin.

First, WLOG, we assume (shrinking $M_1$ if necessary) that our set of $n+1$ masses are bifurcating, with both a standard and nonstandard isoperimetric solution. Identify the unique outer endpoints associated with mass $M_i$, and name them $a_i$ (in the case of the standard solution) and $b_i$ (in the case of the nonstandard solution). Assume our solutions are oriented so that the largest region (associatd with $M_{n+1}$) appears on the positive side of the origin in both configurations.

The standard solution will have regions (and endpoints $a_i$) that alternate from the positive to the negative side as $i$ increases. The nonstandard solution, by virtue of being nonstandard, will not -- therefore it will need to break this pattern at some point. There will need to be a largest pair of consecutive indices $j$, $j+1$ for which $b_j$ and $b_{j+1}$ are both on the same side of the origin. Our goal will be to use the intervals associated with $M_j$ and $M_{j+1}$ to construct a contradiction.

An overview of our argument is as follows: We aim to perform a direct comparison, but only after we carefully ``inflate'' the masses larger than $M_{j+1}$ so that they each have a partner of equal mass on the opposite side of the origin. We will inflate $M_n$ to be equal to $M_{n+1}$; inflate $M_{n-2}$ to be equal to $M_{n-1}$; and so on, until we get to the offending index. We will show that doing this carefully will allow us to keep both a standard and nonstandard solution. Making the masses pairwise equal will simplify the direct comparison of endpoints that we hope to do later.

Here, we include more details about ``inflating'' appropriate masses to reach pairs of equal masses. We start with our largest indices. Assume $M_{n}$ is strictly smaller than $M_{n+1}$. Additionally, assume that $b_n$ and $b_{n+1}$ have opposite signs. (If they had the same sign, then we would already have our consecutive indices with $j=n$.) Our previous work shows that our nonstandard configuration is shifted to the left of the standard configuration, so that $b_n < a_n < 0$. In this case, $b_n$ is further away from the origin than $a_n$. This means, as in Proposition \ref{PropOuterMassGrowth}, adding any amount of mass to $M_n$ will shift $b_n$ and $a_n$ further to the left, resulting in new points $b_n'$ and $a_n'$ that satisfy $f(b_n') - f(b_n) < f(a_n') - f(a_n)$. 
We conclude that our nonstandard configuration has less total perimeter than the standard configuration. We have shown

\noindent \textbf{Partial Result}: If there exists a set of $n+1$ masses $M_1 \leq \dots \leq M_n+1$ for which a nonstandard framework is isoperimetric, and that isoperimetric framework has $M_n$ and $M_{n+1}$ on opposite sides of the origin, then there exists a (possibly new) set of $n+1$ masses, also with a nonstandard isoperimetric configuration, in which the two largest masses are of equal size.

\begin{remark}
We can, if necessary, shrink our smallest mass until our set of masses is once again a bifurcating set. 
\end{remark}

Next, we observe that when $M_n = M_{n+1}$, we can WLOG flip our nonstandard configuration across the origin to put the third largest region ($M_{n-1}$) on the side of the positive $x$-axis (this will already be true for the standard configuration). By our assumption that $n$-bubbles are always standard, we have that (newly flipped) nonstandard configuration is still shifted to the left of the standard configuration.

We now perform a similar analysis on the third- and fourth-largest regions, $M_{n-1}$ and $M_{n-2}$. Again, we assume that $b_{n-2}$ is on the negative side of the axis (or else we would have our index $j$). We find that, as before, we can ``inflate'' the size of $M_{n-2}$ until it reaches the same size as $M_{n-1}$. This inflation would move the leftmost endpoints of each configuration further to the left, in a manner that ultimately leads the nonstandard configuration to have less perimeter than the standard configuration. The details are shown in Figures \ref{Figure8IntervalInflation} and \ref{Figure9ShiftingLeft}. Notice that the four intervals $[b_n', b_{n}]$, $[b_{n-2}', b_{n-2}]$, $[a_n', a_{n}]$, and $[a_{n-2}', a_{n-2}]$ all have the same weighted mass (which is exactly the difference between $M_{n-2}$ and $M_{n-1}$). The left-shifted nature of our non-standard configuration, and Proposition \ref{PropInequality2}, implies that $f(b_n') - f(b_{n}) < f(a_n') - f(a_{n})$ and $f(b_{n-2}') - f(b_{n-2}) < f(a_{n-2}') - f(a_{n-2})$. This will immediately imply that, after inflating $M_{n-2}$ and shifting our endpoints appropriately, our new perimeter is lower in the nonstandard example. We have shown

\noindent \textbf{Partial Result}: If there exists a set of $n+1$ masses $M_1 \leq \dots \leq M_n+1$ for which a nonstandard framework is isoperimetric, and that isoperimetric framework has $M_{n-1}$ and $M_{n+1}$ on one side of the origin opposite $M_n$ and $M_{n-2}$, then there exists a (possibly new) set of $n+1$ masses, also with a nonstandard isoperimetric configuration, in which $M_{n+1} = M_n$ and $M_{n-2} = M_{n-1}$.

%
%
\begin{figure}[H]
    \centering
    \begin{subfigure}[h]{0.45\textwidth}
    \begin{tikzpicture}[scale=0.65]
    \draw[gray, stealth-,very thick] (-1,1) -- (0,0);
    \draw[gray, -stealth,very thick](0,0) -- (1,1);
    \draw[gray, stealth-stealth,very thick] (-4,0) -- (4,0); 
    
    
    \filldraw[blue,ultra thick] (-1.5,0) -- (2.0,0);
    \filldraw[green,ultra thick] (-2.0,0) -- (-1.5,0);

    \filldraw[red,ultra thick] (-2.5,0) -- (-2.0,0);
	\filldraw[yellow, ultra thick] (2.0,0) -- (2.4,0);
	\filldraw[red, ultra thick] (2.4,0) -- (2.9,0);	
    

    \filldraw (-2.5,0) circle (2pt);
   	\filldraw (-2.0,0) circle (2pt);
    \draw[<-] (-2.0,-0.3) -- (-2.0, -0.7) node[anchor=north] {\tiny $a_{n-2}$};
    \filldraw (2.0,0) circle (2pt);
	\filldraw (-1.5,0) circle (2pt); 
	\filldraw (2.4, 0) circle (2pt);   
	\filldraw (2.9, 0) circle (2pt);

    \end{tikzpicture}
    	\end{subfigure}
 	\hfill
    \begin{subfigure}[h]{0.45\textwidth}    
        
  \begin{tikzpicture}[scale=0.65]
    \draw[gray, stealth-,very thick] (-1,1) -- (0,0);
    \draw[gray, -stealth,very thick](0,0) -- (1,1);
    \draw[gray, stealth-stealth,very thick] (-4,0) -- (4,0); 
    
    \filldraw[red,ultra thick] (-2.2,0) -- (-2.7,0);
    \filldraw[yellow, ultra thick] (1.8,0) -- (2.3,0);
    \filldraw[red, ultra thick] (2.3,0) -- (2.7,0);
    
    \filldraw[blue,ultra thick] (-1.7,0) -- (1.8,0);
    \filldraw[green,ultra thick] (-2.2,0) -- (-1.7,0);


   	\filldraw (-2.2,0) circle (2pt);
   	\draw[<-] (-2.2,-0.3) -- (-2.2, -0.7) node[anchor=north] {\tiny $b_{n-2}$};
   	\filldraw (-2.7,0) circle (2pt);
    \filldraw (1.8,0) circle (2pt);
	\filldraw (-1.7,0) circle (2pt); 
	\filldraw (2.3, 0) circle (2pt);   
	\filldraw (2.7, 0) circle (2pt);   

    \end{tikzpicture}    
    
    	\end{subfigure}
    
    \caption{Here we see two configurations with the same weighted perimeter. The left picture is our standard configuration, while the right is a nonstandard configuration. In each picture, the outermost (red) intervals contain the same weighted area. In each picture, the yellow region contains mass that is slightly bigger than the green region. To keep the picture clear, all smaller intervals are omitted and shown as blue.}
    \label{Figure8IntervalInflation}
\end{figure}
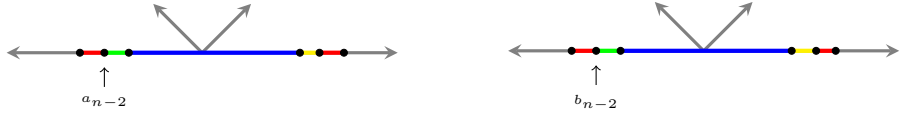

%
%
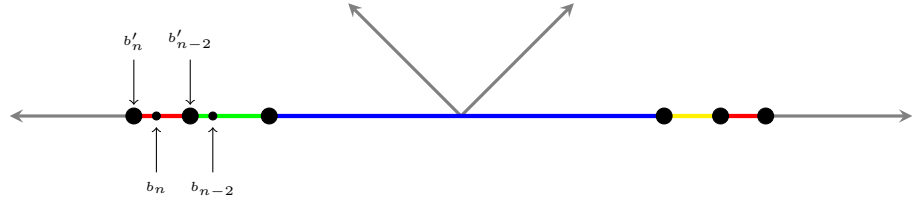
\begin{figure}[H]
    \centering
        
  \begin{tikzpicture}[scale=1.5]
    \draw[gray, stealth-,very thick] (-1,1) -- (0,0);
    \draw[gray, -stealth,very thick](0,0) -- (1,1);
    \draw[gray, stealth-stealth,very thick] (-4,0) -- (4,0); 
    
    \filldraw[red,ultra thick] (-2.4,0) -- (-2.9,0);
    \filldraw[yellow, ultra thick] (1.8,0) -- (2.3,0);
    \filldraw[red, ultra thick] (2.3,0) -- (2.7,0);
    
    \filldraw[blue,ultra thick] (-1.7,0) -- (1.8,0);
    \filldraw[green,ultra thick] (-2.4,0) -- (-1.7,0);


   	\draw[<-] (-2.2,-0.1) -- (-2.2, -0.5) node[anchor=north] {\tiny $b_{n-2}$};
   	\draw[<-] (-2.7,-0.1) -- (-2.7, -0.5) node[anchor=north] {\tiny $b_{n}$};
   	\draw[<-] (-2.9,0.1) -- (-2.9, 0.5) node[anchor=south] {\tiny $b_{n}'$};
   	\draw[<-] (-2.4,0.1) -- (-2.4, 0.5) node[anchor=south] {\tiny $b_{n-2}'$};

   	\filldraw (-2.2,0) circle (1pt);
    \filldraw (-2.9,0) circle (2pt);
    \filldraw (-2.4,0) circle (2pt);
   	\filldraw (-2.7,0) circle (1pt);
    \filldraw (1.8,0) circle (2pt);
	\filldraw (-1.7,0) circle (2pt); 
	\filldraw (2.3, 0) circle (2pt);   
	\filldraw (2.7, 0) circle (2pt);   

    \end{tikzpicture}    
     \caption{Here, we see our nonstandard configuration with $M_{n-2}$ inflated until the green interval has the same weighted mass as the yellow interval. This has resulted in two of the original endpoints ($b_n$ and $b_{n-2}$) shifting further to the left. (The new endpoints are denoted $b_n'$ and $b_{n-2}'$.) A similar shift occurs in the standard configuration, with new endpoints $a_{n}'$ and $a_{n-2}'$.}
     \label{Figure9ShiftingLeft}
\end{figure}

We see that we can continue this process. At each step, we begin with a bifurcation in which the outermost intervals, in pairs, have the same enclosed mass. We then repeat the following actions:
\begin{itemize}
\item Flip our nonstandard configuration (if necessary) to put the region with the next-largest mass, $M_k$, on the positive side.
\item Observe that our nonstandard configuration is still shifted to the left of our standard one.
\item If the interval with $M_{k-1}$ is on the negative side, we can ``inflate'' $M_{k-1}$ until it is the same size as $M_k$. Due to the left shift and Proposition \ref{PropInequality2}, our nonstandard configuration will have total perimeter no more than than our standard configuration, meaning that the standard configuration is not solely isoperimetric.
\item Shrink $M_1$, if necessary, until our set of masses is once again bifurcating (only now with $M_{k-1}$ set equal to $M_k$) as well.
\end{itemize}

We repeat as necessary until our we reach consecutive indices $j$, $j+1$ with endpoints that are on the same (positive) side of the origin in our nonstandard configuration. Note that the outer regions are grouped into pairs of equal weighted volume. At this point, we have a nonstandard configuration that looks like Figure \ref{Figure10NonstandardCloseup}

%
%
\begin{figure}[H]
    \centering
    \begin{tikzpicture}
    \draw[gray, stealth-stealth,very thick] (-4.5,0) -- (4.5,0); 
    \draw[gray, stealth-,very thick] (-2.5,2.5) -- (0,0);
    \draw[gray, -stealth,very thick](0,0) -- (2.5,2.5);


    \filldraw (-4.1,0) circle (2pt) ;  
    \filldraw (-3.4,0) circle (2pt) ;
    \filldraw (-2.6,0) circle (2pt) node[anchor=north] {$c$} ;
    \filldraw (-2,0) circle (2pt) ;
    
	\node[anchor=south] at (-0.5,0) {\tiny $M_2$};
	\node[anchor=south] at (-1.5,0) {. . .};
   	\draw[<-] (-2.3,0.1) -- (-2.3, 0.9) node[anchor=south] {\tiny $M_{\ell}$};
	\node[anchor=south] at (-3,0) {. . . };
    \filldraw (-1,0) circle (2pt) ;

    \filldraw (0.8,0) circle (2pt) ;
	\node[anchor=south] at (0.5,0) {\tiny $M_1$};
	\node[anchor=south] at (1.1,0) {\tiny $M_3$};
    \filldraw(1.4,0) circle(2pt) ;
    \node[anchor=south] at (1.6,0) {...};
    \filldraw(1.8,0) circle(2pt) ;
	\draw[<-] (2.05,0.1) -- (2.05, 0.4) node[anchor=south] {\tiny $M_{j}$};
    \filldraw (2.3,0) circle (2pt) node[anchor=north] {$d$} ;;
   	\draw[<-] (2.5,0.1) -- (2.5, 0.9) node[anchor=south] {\tiny $M_{j+1}$};
	\filldraw (2.7,0) circle (2pt) ;
	\node[anchor=south] at (3,0) {...};
    \filldraw (3.3,0) circle (2pt) ;
	\node[anchor=south] at (3.8,0) {\tiny $M_{n+1}$};    
	\filldraw (4.2,0) circle (2pt) ;
	\node[anchor=south] at (-3.7,0) {\tiny $M_{n}$};    

    \end{tikzpicture}
    \caption{The $n+1$ regions are alternating for large indices, and these large indices are grouped in consecutive pairs with equal volume ($M_{n+1} = M_n$, etc). The largest index that fails this pattern in $j+1$, so that $M_j$ and $M_{j+1}$ are on the same side of the origin. The largest index, smaller than $j$, that appears on the other side of the origin is $\ell$.}
    \label{Figure10NonstandardCloseup}
\end{figure}
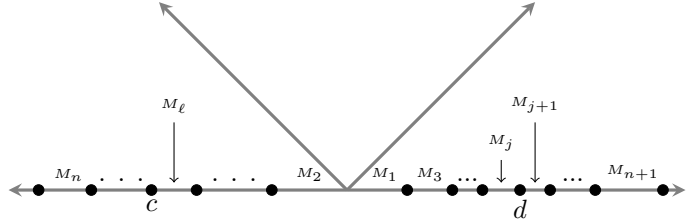       


Let $c$ be the outer endpoint of the interval associated with region $M_\ell$, and let $d$ be the endpoint separating $M_j$ from $M_{j+1}$ (these endpoints have been identified in Figure \ref{Figure10NonstandardCloseup}). In short order, we will show that $|d| < |c|$. We momentarily assume this, and additionally recognize that $M_\ell < M_j$ (it is known that $M_\ell \leq M_j$, and if $M_\ell = M_j$ we can simply change their index names so that $M_j$, $M_{j+1}$ appear on opposite sides of the origin). Then, as seen in Figure \ref{Figure11Transpose}, transposing the regions associated with $M_\ell$ and $M_j$ will result in lower total perimeter. 


\begin{figure}[H]
    \centering
    \begin{tikzpicture}
    \draw[gray, stealth-stealth,very thick] (-4.5,0) -- (4.5,0); 
    \draw[gray, stealth-,very thick] (-2.5,2.5) -- (0,0);
    \draw[gray, -stealth,very thick](0,0) -- (2.5,2.5);


    \filldraw (-4.1,0) circle (2pt) ;  
    \filldraw (-3.4,0) circle (2pt) ;
    \filldraw (-2.6,0) circle (2pt) node[anchor=north] {$c$} ;
    \filldraw (-2,0) circle (2pt) ;
    
	\node[anchor=south] at (-0.5,0) {\tiny $M_2$};
	\node[anchor=south] at (-1.5,0) {. . .};
   	\draw[<-] (-2.3,0.1) -- (-2.3, 0.9) node[anchor=south] {\tiny $M_{\ell}$};
	\node[anchor=south] at (-3,0) {. . . };
    \filldraw (-1,0) circle (2pt) ;

    \filldraw (0.8,0) circle (2pt) ;
	\node[anchor=south] at (0.5,0) {\tiny $M_1$};
	\node[anchor=south] at (1.1,0) {\tiny $M_3$};
    \filldraw(1.4,0) circle(2pt) ;
    \node[anchor=south] at (1.6,0) {...};
    \filldraw(1.8,0) circle(2pt) ;
	\draw[<-] (2.05,0.1) -- (2.05, 0.4) node[anchor=south] {\tiny $M_{j}$};
    \filldraw (2.3,0) circle (2pt) node[anchor=north] {d} ;;
   	\draw[<-] (2.5,0.1) -- (2.5, 0.9) node[anchor=south] {\tiny $M_{j+1}$};
	\filldraw (2.7,0) circle (2pt) ;
	\node[anchor=south] at (3,0) {...};
    \filldraw (3.3,0) circle (2pt) ;
	\node[anchor=south] at (3.8,0) {\tiny $M_{N+1}$};    
	\filldraw (4.2,0) circle (2pt) ;
	\node[anchor=south] at (-3.7,0) {\tiny $M_{N}$};    

\path[draw,decorate,decoration=brace] (-2.0,-0.4) -- (-2.6,-0.4)
node[midway,below,font=\small\sffamily]{};

\path[draw,decorate,decoration=brace] (2.3,-0.4) -- (1.8,-0.4)
node[midway,below,font=\small\sffamily]{};

\draw [<->] (2,-0.7) to [out=-130,in=-50] (-2.2,-0.7);

    \end{tikzpicture}
    \caption{Transposing the locations of the $j$th and $\ell$th regions will result in a decreasing of perimeter. Endpoint $c$ and those to the left of $c$ will increase, but this increase in perimeter will be offset by the decrease in perimeter by $d$ and points to the left of $d$.}
    \label{Figure11Transpose}
\end{figure}
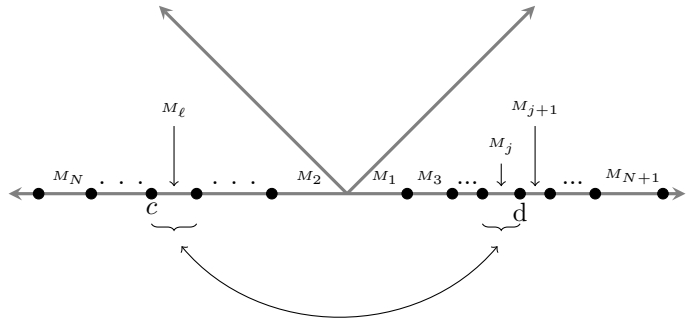

The decreased perimeter can be seen as follows. First, we recognie that swapping the regions $M_j$ and $M_\ell$ will result in a shifting of a certain number of endpoints. The endpoints that are less than or equal to $c$ will shift to the left, resulting in an increased contribution to perimeter. Meanwhile, the endpoints that are greater than or equal to $d$ will also shift to the left, but will result in a decrease in perimeter. (All endpoints are known to shift to the left because $M_\ell < M_j$.) However, by counting regions, we see that there is one more endpoint  moving on the right than on the left. This extra endpoint should be seen as $d$, because all other shifting endpoints can be paired with a partner across the origin. Because $|c| > |d|$, because we inflated masses to be paired and equal on opposite sides of the origin, and using Proposition \ref{PropInequality2}, we see that the decrease in perimeter from each endpoint on the right outweights the increase in perimeter from their partner on the left. Taken together, we see that the total perimeter decreases after the transposition of $M_j$ and $M_\ell$. Thus, the original nonstandard configuration could not have been isoperimetric, giving us out contradiction.

Thus, to complete the argument, it remains to be seen that $|d| < |c|$. This is shown in Figure \ref{Figure12CompareCD} and uses the fact that our nonstandard configuration is left-shifted when compared to our standard configuration.

%
%
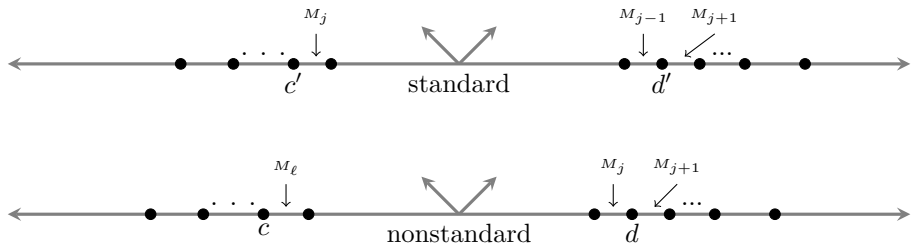
\begin{figure}[H]
    \centering
    \begin{tikzpicture}
    \draw[gray, stealth-stealth,very thick] (-6,0) -- (6,0); 
    \draw[gray, stealth-,very thick] (-0.5,0.5) -- (0,0);
    \draw[gray, -stealth,very thick](0,0) -- (0.5,0.5);

    \filldraw (-4.1,0) circle (2pt) ;  
    \filldraw (-3.4,0) circle (2pt) ;
    \filldraw (-2.6,0) circle (2pt) node[anchor=north] {$c$} ;
    \filldraw (-2,0) circle (2pt) ;
   	\draw[<-] (-2.3,0.1) -- (-2.3, 0.4) node[anchor=south] {\tiny $M_{\ell}$};
	\node[anchor=south] at (-3,0) {. . . };

    \filldraw(1.8,0) circle(2pt) ;
	\draw[<-] (2.05,0.1) -- (2.05, 0.4) node[anchor=south] {\tiny $M_{j}$};
	\draw[<-] (2.6,0.1) -- (2.9, 0.4) node[anchor=south] {\tiny $M_{j+1}$};
    \filldraw (2.3,0) circle (2pt) node[anchor=north] {$d$} ;
	\filldraw (2.8,0) circle (2pt) ;
	\node[anchor=south] at (3.1,0) {...};
    \filldraw (3.4,0) circle (2pt) ;
	\filldraw (4.2,0) circle (2pt) ;   
	
	\node[anchor=north] at (0,0) {nonstandard};

	    \draw[gray, stealth-stealth,very thick] (-6,2) -- (6,2); 
    \draw[gray, stealth-,very thick] (-0.5,2.5) -- (0,2);
    \draw[gray, -stealth,very thick](0,2) -- (0.5,2.5);

    \filldraw (-3.7,2) circle (2pt) ;  
    \filldraw (-3.0,2) circle (2pt) ;
    \filldraw (-2.2,2) circle (2pt) node[anchor=north] {$c'$} ;
    \filldraw (-1.7, 2) circle (2pt) ;
   	\draw[<-] (-1.9,2.1) -- (-1.9, 2.4) node[anchor=south] {\tiny $M_{j}$};
	\node[anchor=south] at (-2.6,2) {. . . };

    \filldraw(2.2,2) circle(2pt) ;
	\draw[<-] (2.45,2.1) -- (2.45, 2.4) node[anchor=south] {\tiny $M_{j-1}$};
	\draw[<-] (3.0,2.1) -- (3.4, 2.4) node[anchor=south] {\tiny $M_{j+1}$};
    \filldraw (2.7,2) circle (2pt) node[anchor=north] {$d'$} ;
	\filldraw (3.2,2) circle (2pt) ;
	\node[anchor=south] at (3.5,2) {...};
    \filldraw (3.8,2) circle (2pt) ;
	\filldraw (4.6,2) circle (2pt) ;   

	\node[anchor=north] at (0,2) {standard};

    \end{tikzpicture}
    \caption{A nonstandard and a standard configuration lined up for comparison. The nonstandard configuration is left-shifted.}
    \label{Figure12CompareCD}
\end{figure}       

Because the nonstandard configuration is left shifted, and because both configurations have the same alternating pattern for the largest regions (i.e. those greater than $j$), we can conclude that $d' > d$ and that $c' > c$. Furthermore, because of the alternating nature of the standard configuration, we know that $|c'| > |d'|$. Taken together, we get 
$
|c| > |c'| > |d'| > |d|
$
 as desired.
 

\bibliographystyle{plain}
\bibliography{2021NBubble}

\end{document}